\documentclass{amsart}
\usepackage{graphicx}
\usepackage{amssymb,amsthm,url}

\usepackage{float}
\floatstyle{boxed} 
\restylefloat{figure}

\makeatletter

\let\c@table\c@figure
\makeatother

\newtheorem{theorem}{Theorem}
\newtheorem{lemma}[theorem]{Lemma}
\newtheorem{corollary}[theorem]{Corollary}

\theoremstyle{definition}
\newtheorem{definition}[theorem]{Definition} 
\newtheorem{construction}[theorem]{Construction}

\newcommand{\EOC}{\hfill $\diamond$}

\newcommand{\C}{\mathrm C}
\newcommand{\D}{\mathrm D}
\newcommand{\V}{\mathrm V}

\newcommand{\Cos}{\mathrm{Cos}}
\newcommand{\Aut}{\mathrm{Aut}}
\newcommand{\Cay}{\mathop{\mathrm{Cay}}}

\newcommand{\Merge}{\mathrm {M}}
\newcommand{\Split}{\mathrm {S}}

\newcommand{\cC}{\mathcal {C}}
\newcommand{\cP}{\mathcal {P}}
\newcommand{\cR}{\mathcal {R}}
\newcommand{\cT}{\mathcal {T}}

\newcommand{\ZZ}{\mathbb Z}

\newcommand{\la}{\langle}
\newcommand{\ra}{\rangle}

\begin{document}

\title{Cubic vertex-transitive graphs on up to $1280$ vertices}

\author[P. Poto\v{c}nik]{Primo\v{z} Poto\v{c}nik}
\address{Primo\v{z} Poto\v{c}nik,\newline
 Faculty of Mathematics and Physics,
 University of Ljubljana, \newline 
Jadranska 19, 1000 Ljubljana, Slovenia}\email{primoz.potocnik@fmf.uni-lj.si}

\author[P. Spiga]{Pablo Spiga}
\address{Pablo Spiga,\newline
 University of Milano-Bicocca, Departimento di Matematica Pura e Applicata, \newline
 Via Cozzi 53, 20126 Milano Italy} \email{pablo.spiga@unimib.it}

\author[G. Verret]{Gabriel Verret}
\address{Gabriel Verret,\newline
Faculty of Mathematics, Nat. Sci. and Info. Tech., University of Primorska, \newline 
Glagolja\v{s}ka 8, 6000 Koper, Slovenia}
\email{gabriel.verret@pint.upr.si}

\thanks{Address correspondence to P. Spiga (pablo.spiga@unimib.it)}

\subjclass[2000]{20B25}
\keywords{cubic, tetravalent, valency $3$, valency $4$, vertex-transitive, arc-transitive} 

\thanks{The authors would like to thank Simon Guest for technical advice and for hosting some of the computations and Marston Guest for providing his list of small regular maps.}

\begin{abstract}
A graph is called \emph{cubic} and \emph{tetravalent} if all of its vertices have valency $3$ and $4$, respectively. It is called \emph{vertex-transitive} and \emph{arc-transitive} if its automorphism group acts transitively on its vertex-set and on its arc-set, respectively. In this paper, we combine some new theoretical results with computer calculations to construct all cubic vertex-transitive graphs of order at most $1280$. In the process, we also construct all tetravalent arc-transitive graphs of order at most $640$.
\end{abstract}

\maketitle

\section{Introduction}\label{sec:Intro}
Throughout this paper, all graphs considered will be finite and simple (undirected, loopless and with no multiple edges). A graph is called \emph{vertex-transitive} if its automorphism group acts transitively on its vertex-set. The family of vertex-transitive graphs has been the subject of much research and there are still many important questions concerning it that are still open and active. Since a disconnected vertex-transitive graph consists of pairwise isomorphic connected components, there is little loss of generality in assuming connectedness, which we will do throughout the paper.

It is easy to see that a vertex-transitive graph must be \emph{regular}, that is, all of its vertices must have the same valency. If this valency is at most $2$ and the graph has order at least $3$, then the graph must be a cycle. 

In this sense, the first non-trivial case is that of \emph{cubic} graphs, that is, regular graphs of valency 3. Hence, many questions about vertex-transitive graphs are first considered in the cubic case (see~\cite{CSS,Glover,KMZ, Li} for example). Some of these questions are still very hard even in the cubic case, at least given our current understanding of cubic vertex-transitive graphs.

Many authors have tried to make some headway by constructing all cubic vertex-transitive graphs of certain type, for example those with order admitting a particularly simple factorisation. Another idea is to determine all cubic vertex-transitive graphs up to a certain order. For example, Read and Wilson enumerated~\cite[pp. 161-163]{ReadWilson} all such graphs of order at most $34$ and McKay and Royle have constructed a table~\cite{McKayRoyle} of cubic vertex-transitive graphs of small order which is complete up to order $94$. To the best of our knowledge, these were the best results available of this kind up to now. In this paper, we construct a census of all cubic vertex-transitive graphs of order at most $1280$. We find that there are $111360$ non-isomorphic such graphs; obviously a list cannot be given in this paper but we present some overall data in Section~\ref{sec:data}. A complete list of these graphs in \texttt{Magma}~\cite{magma} code can be found online at~\cite{Census3}. We now give an overview of the methods we employed but we first need to fix some more basic terminology and discuss a few basic results. 

A graph $\Gamma$ is said to be $G$-\emph{vertex-transitive} if $G$ is a subgroup of $\Aut(\Gamma)$ acting transitively on the vertex-set $\V\Gamma$ of $\Gamma$. Similarly, $\Gamma$ is said to be $G$-\emph{arc-transitive} if $G$ acts transitively on the arcs of $\Gamma$ (that is, on the ordered pairs of adjacent vertices of $\Gamma$). When $G=\Aut(\Gamma)$, the prefix $G$ in the above notation is sometimes omitted.

Let $\Gamma$ be a cubic $G$-vertex-transitive graph, let $v$ be a vertex of $\Gamma$ and let $m$ be the number of orbits of the vertex-stabiliser $G_v$ in its action on the neighbourhood $\Gamma(v)$. It is an easy observation that, since $\Gamma$ is $G$-vertex-transitive, $m$ is equal to the number of orbits of $G$ in its action on the arcs of $\Gamma$ (and, in particular, does not depend on the choice of $v$). Since $\Gamma$ is cubic, it follows that $m\in\{1,2,3\}$ and there is a natural split into three cases, according to the value of $m$. This split into three cases was considered already in~\cite{Zero, Lorimer} when $G=\Aut(\Gamma)$.

If $m=1$, then $\Gamma$ is $G$-arc-transitive. This case is by far the easiest to deal with. The so-called Foster Census, an extensive list of examples of small order, was started as early as 1932, has been much extended since and is now known to be complete up to at least order $2048$ (see~\cite{FosterBouwer, Conder2048, ConderFosterCensus, Foster}). The arc-transitive case was thus already dealt with but we nevertheless briefly explain the method used and why it cannot be immediately applied to the general vertex-transitive case. The method used in the arc-transitive case relies on a celebrated theorem of Tutte~\cite{Tutte,Tutte2} which shows that, if $\Gamma$ is a cubic $G$-arc-transitive graph, then the vertex-stabiliser $G_v$ has order at most $48$ and hence $|G|\leq 48|\V(\Gamma)|$. Since the order of the groups involved grows at most linearly with the order of the graphs and the groups have a particular structure~\cite{DjoMi}, a computer algebra system can find all the graphs up to a certain order rather efficiently (by using the \texttt{LowIndexNormalSubgroups} algorithm in \texttt{Magma} for example).

If $m=3$, then $G_v$ fixes the neighbours of $v$ pointwise and, by connectedness, it is easily seen that $G_v=1$. This lack of structure of the vertex-stabiliser makes it difficult to use the  method that was successfully used in the arc-transitive case. On the other hand, since $G_v=1$, it follows that $|G|=|\V(\Gamma)|\leq 1280$. This allows us to use the \texttt{SmallGroups} database in \texttt{Magma} to find all possibilities for $G$ (and then for $\Gamma$). There are some tricks involved which help remarkably in restricting the search space and making the computation feasible; the details are in Section~\ref{sec:Cayley}. 

The final and probably hardest case is when $m=2$. A recent attempt~\cite{MW} at classifying the small graphs in this case was only done in a special case and only up to order $150$. The main difficulty is that the vertex-stabiliser can have arbitrarily large order; for every $n\geq 5$, there exists a cubic vertex-transitive graph of order $4n$ and with vertex-stabiliser of order $2^{n-1}$. In Section~\ref{LocallyC2} we show that, in this case, we can construct an auxiliary graph which is tetravalent, $G$-arc-transitive and has half the order of $\Gamma$. Moreover, we show that this construction can be reversed. Therefore, in order to find all cubic $G$-vertex-transitive graphs with $m=2$ up to $n$ vertices, it suffices to construct the list of all tetravalent arc-transitive graphs of order at most $n/2$. Recently, the authors of this paper have proved~\cite{PSV4valent} that, apart from a family of well-understood exceptions, the order of the automorphism group of a tetravalent arc-transitive graph is bounded above by a quadratic function of the order of the graph. This allows us to use a method similar to the one used in the cubic arc-transitive case to construct a list of all tetravalent arc-transitive graphs of order at most $640$; the details are in Section~\ref{4valent}.

\section{Data about the graphs}\label{sec:data}

There are too many graphs in our census to discuss them individually in this paper. The complete list of graphs in \texttt{Magma} code can be found online~\cite{Census3}, while here we simply present some aggregate data about them and give an overall impression. 

Figure~\ref{graph:total} shows the number of cubic vertex-transitive graphs of order at most $n$ with respect to $n$ (these are the black data points). Superimposed on this data (in gray) is the graph of the function $n \mapsto n^2/15$, which seems to be a close approximation on the range considered. We do not expect this trend to continue. In fact, in an upcoming paper~\cite{Asymptotic}, we prove that if $f(n)$ is the number of cubic vertex-transitive graphs of order at most $n$, then $\log(f(n))\in\Theta((\log n)^2)$. In other words, there exist positive constants $c_1$ and $c_2$ such that, for every large enough $n$, $$c_1(\log n)^2\leq\log(f(n))\leq c_2(\log n)^2.$$

\begin{center}
\begin{figure}[h]
\includegraphics[scale=0.80]{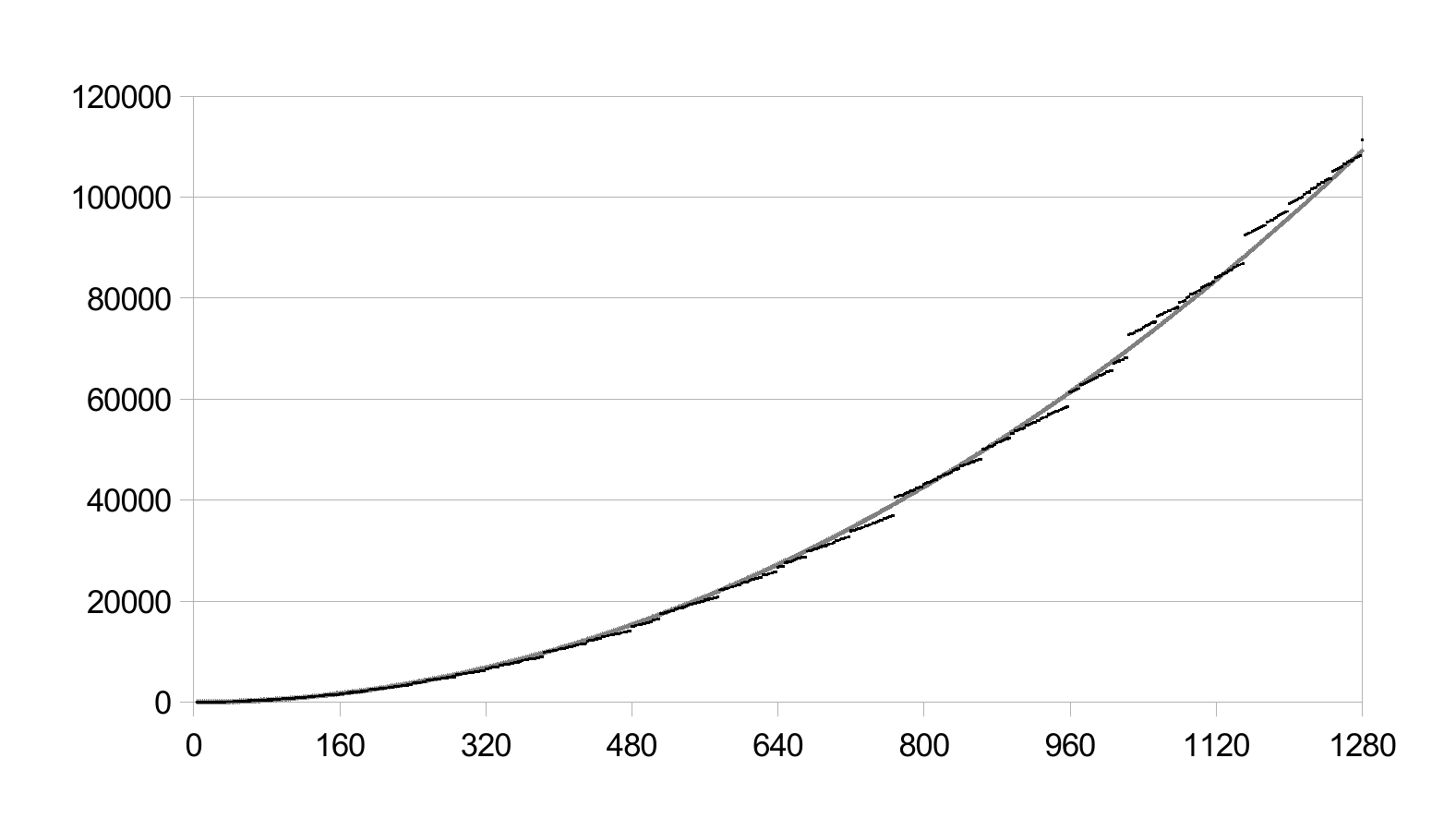}
\caption{\small{Number of cubic vertex-transitive graphs of order at most $n$.}}
\label{graph:total}
\end{figure}
\end{center}

Also note that, while the data points in Figure~\ref{graph:total} can be approximated rather well by a smooth function, some larger jumps can be distinctly seen where $n$ has lots of small prime factors (for example, at $768$, $1024$ and $1152$).

In Table~\ref{table:cubic}, we give the number of cubic vertex-transitive graphs of order at most $1280$, discriminating with respect to whether the graphs are Cayley or not and with respect to the parameter $m$, the number of arc-orbits of the automorphism group. Thus, the column labelled ``$m=1$'' corresponds to arc-transitive graphs, while the column labelled ``$m=3$'' corresponds to graphs the automorphism groups of which act regularly on their vertices (such graphs are sometimes called \emph{graphical regular representations}, or GRRs, and are necessarily Cayley, as noted in the introduction).

\begin{table}[h]
\begin{center}
\begin{tabular}{|c|c|c|c|c|c|}\hline
 & $m=1$ & $m=2$ & $m=3$ & Total \\\hline
Cayley &386  &11853&97687& 109926\\
Non-Cayley &96  &1338&0& 1434\\
Total & 482 &13191&97687& 111360\\\hline
\end{tabular}
\medskip
\caption{\small{Number of cubic vertex-transitive graphs of order at most $1280$.}}\label{table:cubic}
\end{center}
\end{table}

As can be seen in Table~\ref{table:cubic}, the vast majority of the graphs in this range are Cayley and, among those, the majority are GRRs. These two facts seem to be part of a trend, as the next figure shows.

\begin{center}
\begin{figure}[h]
\includegraphics[scale=0.80]{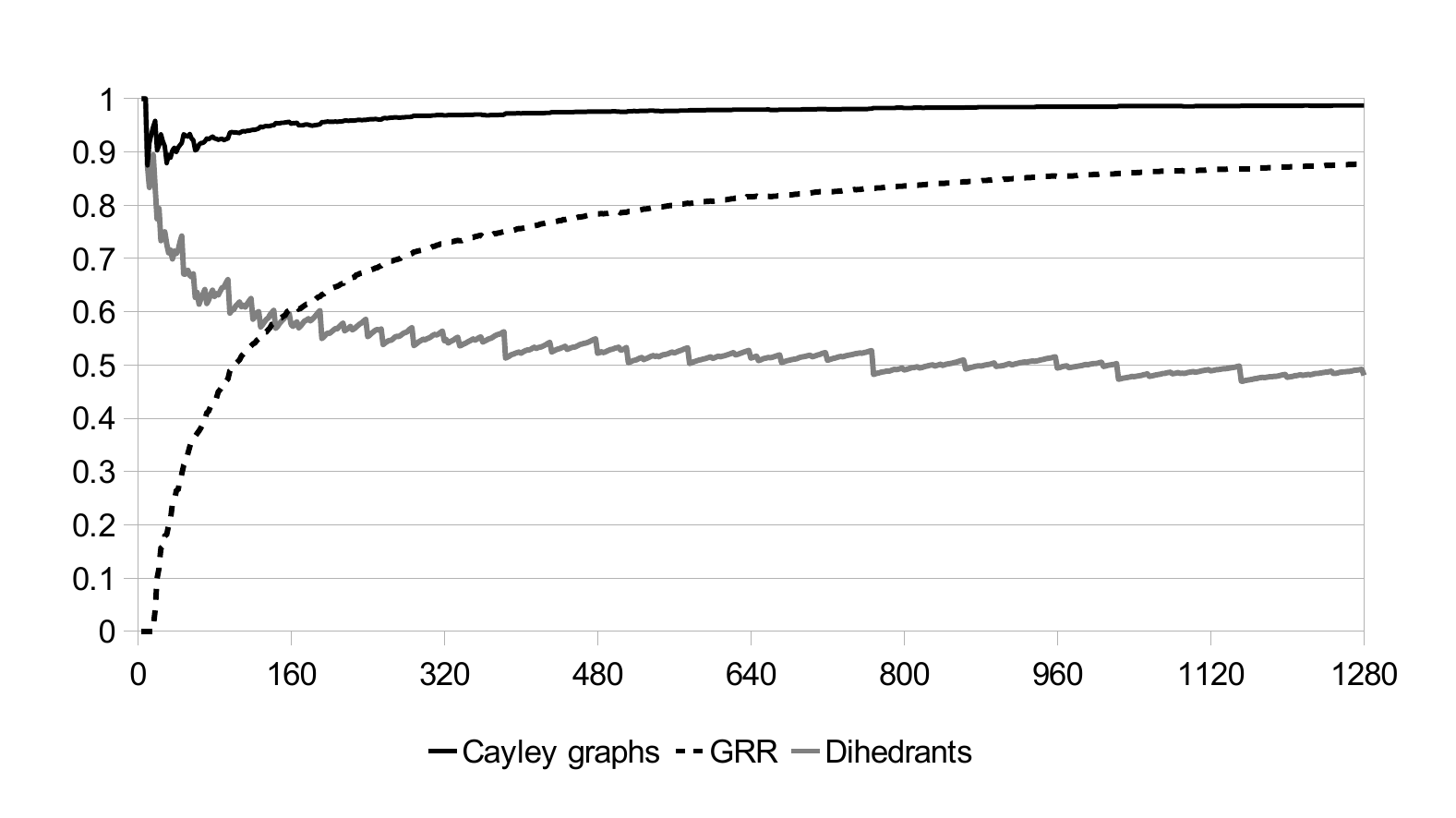}
\caption{\small{Proportion of graphs of different types of order at most $n$.}}
\label{graph:Percent}
\end{figure}
\end{center}

Figure~\ref{graph:Percent} shows the proportion of Cayley graphs, GRRs and dihedrants among all the cubic vertex-transitive graphs of order at most $n$ as a function of $n$ (a \emph{dihedrant} is a Cayley graph on a dihedral group). 

The proportion of GRRs seems to tend towards $1$. It is conjectured that almost all vertex-transitive graphs are Cayley graphs~\cite{McKayPraeger} and that almost all Cayley graphs are GRRs (see~\cite{BabaiGodsil} for a variant of this conjecture). Our data might suggest that these conjectures are also true when we restrict our attention to cubic graphs. 

It is also interesting that the proportion of dihedrants is initially very large and seems to stabilise at about one-half near the end of the range under consideration. However, this proportion actually tends to $0$ since the number of cubic dihedrants of order at most $n$ can grow at most polynomially with $n$ while the total number of vertex-transitive graphs of order at most $n$ grows faster than any polynomial function of $n$(this will be proved in~\cite{Asymptotic}).

We tested all the graphs in the census for Hamilton cycles and found at least one such cycle for each graph apart from the four well-known exceptions (the Petersen graph, the Coxeter graph and their truncations).

We also computed the diameter and girth of each of these graphs to see if they would yield new extremal examples to either the degree diameter problem (see~\cite{LozSiran,MillSiran}) or the cage problem (see~\cite{ExooJaj}). Let $n_{cay}(k,g)$ ($n_{vt}(k,g)$, respectively) denote the order of the smallest Cayley (vertex-transitive, respectively) graph of valency $k$ and girth $g$ and let $m_{cay}(k,d)$ ($m_{vt}(k,d)$, respectively) denote the order of the largest Cayley (vertex-transitive, respectively) graph of valency $k$ and diameter $d$. 

Simply by going through all the graphs in our census and computing their girth, one can see that the values of $n_{cay}(3,g)$ and  $n_{vt}(3,g)$ for $g\le 16$ are as shown in Table~\ref{table:girth} below. Similarly, by computing the diameter and using the fact that the number of vertices of a cubic graph of diameter $d$ cannot exceed $3\cdot 2^d - 2$, we obtain the values of $m_{cay}(3,d)$ and  $m_{vt}(3,d)$ for $d\le 8$ and bounds on these when $9\leq d \leq 12$.

\begin{table}[h]
\begin{center}
\begin{tabular}{cc}
\begin{tabular}{|c|c|c|}
\hline
$g$ &   $ n_{cay}(3,g)$ &   $n_{vt}(3,g)$  \\
\hline\hline
$3$ &        $4$   &             $4 $ \\
$4$  &        $6$        &  $        6$ \\
$5$   &       $50$      &  $             10$ \\
$6$   &       $14$         &  $          14$ \\
$7$   &       $30$     &  $             26$ \\
$8$   &       $42$     &  $              30$ \\
$9$   &       $60$   &  $               60$ \\
$10$   &     $96$     &  $             80$ \\
$11$    &    $192$      &  $            192$ \\
$12$     &   $162$      &  $            162$ \\
$13$    &    $272$      &  $            272$ \\
$14$    &     $406$     &  $             406$ \\
$15$    &    $864$       &  $           620$ \\
$16$    &  $1008$      &  $            1008$ \\
\hline
\end{tabular}

&
\begin{tabular}{c}
\begin{tabular}{|c|c|c|}
\hline
$d$ &$ m_{cay}(3,d)$ &   $m_{vt}(3,d)$ \\
\hline\hline
$2$ & $8$ & $ 10$\\
$3$ & $14$ & $14$\\
$4$  & $24$ & $ 30$\\
$5$  & $60$ & $60$\\
$6$  & $72$ & $82$\\
$7$  & $168$ & $168$\\
$8$  & $300$ & $300$ \\
$9$  & $\ge 506$ & $\ge 546$\\
$10$  & $\ge 882$ & $\ge 1250$ \\
$11$  & $\ge 1250$ & $\ge 1250$\\
$12$  & $\ge 1250$ & $\ge 1250$ \\
\hline
\end{tabular} \\  \\ \\ \\
\end{tabular}
\end{tabular}
\medskip
\caption{\small{Extremal order given the girth or diameter.}}\label{table:girth}
\end{center}
\end{table}

\section{$m=3$ (Cayley graphs)}\label{sec:Cayley}
The following definition will be very important in this section.
\begin{definition}
Let $G$ be a group and let $S$ be an inverse-closed and identity-free subset of $G$. The \emph{Cayley graph} $\Cay(G,S)$ has vertex-set $G$ and two vertices $u$ and $v$ are adjacent if and only if  $uv^{-1}\in S$.
\end{definition}
Note that, because of the restrictions imposed on $S$, all Cayley graphs in this paper will be simple. It is easy to see that $\Cay(G,S)$ is vertex-transitive and $|S|$-regular. Moreover, $\Cay(G,S)$ is connected if and only if $S$ generates $G$.

Our goal in this section is to construct all cubic graphs $\Gamma$ of order at most $1280$ which admit a vertex-transitive group $G$ such that $G_v=1$. The fact that $G_v=1$ implies that $G$ acts regularly on the vertices of $\Gamma$ and, by a result of Sabidussi~\cite{Sabidussi}, it follows that $\Gamma\cong\Cay(G,S)$ for some $S$.

It would therefore suffice to consider all groups $G$ of order at most $1280$, then, for each such group, to find all inverse-closed and identity-free $3$-subsets of $G$ and, finally, to test for isomorphism of the resulting graphs. This extremely naive method is not computationally feasible. Fortunately, a few simple ideas will greatly improve the situation. We start with two elementary observations, which are the main key to our approach. (We denote by $\ZZ_n$ the cyclic group of order $n$ and, for a group $G$, we denote its commutator subgroup by $G'$.)

\begin{lemma}\label{1}Let $G$ be a finite group and let $\Cay(G,S)$ be a connected Cayley graph of valency at most $3$. Then $G/G'$ is isomorphic to one of $\ZZ_2\times \ZZ_2\times \ZZ_2$, $\ZZ_2\times \ZZ_r$ (for some $r\geq 2$), or $\ZZ_r$ (for some $r\geq 1$).
\end{lemma}
\begin{proof}Since $\Cay(G,S)$ is connected and has valency at most $3$, it follows that $G$ has an inverse-closed set of generators of cardinality at most $3$ and hence so does the quotient group $G/G'$. The conclusion then follows from the fact that $G/G'$ is abelian.
\end{proof}

We denote by $\mathcal{P}_n$ the class of groups $G$ such that $G$ has order $n$ and $G/G'$ is isomorphic to one of $\ZZ_2\times \ZZ_2\times \ZZ_2$, $\ZZ_2\times \ZZ_r$, or $\ZZ_r$. In view of Lemma~\ref{1}, we may assume that $G$ is in $\mathcal{P}_n$ for some $n\leq 1280$. This fairly naive consideration drastically reduces the number of groups we have to consider. (For example, there are $1090235$ non-isomorphic groups of order $768$, but only $4810$ are in $\mathcal{P}_{768}$.) The \texttt{SmallGroups} database in \texttt{Magma} has a built-in filter that allows us to quickly construct $\mathcal{P}_n$ for $n\leq 1280$. Our second ingredient is the following old observation (see~\cite{Babai} for example).

\begin{lemma}\label{2} 
Let $G$ be a group, let $\varphi$ be an automorphism of $G$ and let $S$ be an inverse-closed and identity-free subset of $G$. Then $\Cay(G,S)\cong \Cay(G,S^\varphi)$.
\end{lemma}

By Lemma~\ref{1} and Lemma~\ref{2}, to construct all connected cubic Cayley graphs on $n$ vertices, it suffices to consider groups in $\mathcal{P}_n$ and, for each such group $G$, to determine the $\Aut(G)$-classes of inverse-closed generating $3$-subsets of $G$.

Computationally, this is already much easier than the very naive method suggested at the beginning of this section.  

The complexity of the algorithms of Eick, Leedham-Green and O'Brian~\cite{ELO} for computing $\Aut(G)$ depends heavily on the number of generators of $G$ (and on the rank of the elementary abelian sections of $G$). Since the group $G$ is $3$-generated, computing $\Aut(G)$ takes a reasonable amount of time. 

However, the whole process is still not computationally trivial and a few more tricks are involved to speed up the computation. The most important one is to convert $G$ to a permutation group in \texttt{Magma} and to work with the holomorph of $G$ rather than $\Aut(G)$. Since \texttt{Magma} returns the holomorph of $G$ as a permutation group with its natural action on $G$, this allows the use of standard \texttt{Magma} functions to compute the orbits of $\Aut(G)$ on the set of involutions of $G$. Moreover, the time required in computing the holomorph of $G$ is more than compensated by the efficiency of the various \texttt{Magma} algorithms for permutation groups.

These improvements are enough to run the \texttt{Magma} computations to completion in a reasonable amount of time. This yields a list of all connected cubic Cayley graphs of a given order. It remains to test for isomorphism to get rid of repetitions. This is another computationally hard process. However, it is feasible in this case because the \texttt{nauty} algorithm of McKay~\cite{{McKaynauty}} is faster on graphs with relatively few edges (for example cubic graphs). (The authors of this paper noticed that replacing our graphs with the output of the \texttt{Magma} command \texttt{StandardGraph} seemed to speed up the isomorphism testing significantly.)

We were able to run this process to completion for every order $n\leq 1280$ except $512$ and $1024$. The list of groups of order $512$ is simply too large for the method above to complete in reasonable time. As for order $1024$, the situation is even worse: there is no exhaustive list of these groups available in the \texttt{Magma} library.

For these two orders, we use a different method which relies on the fact that $G$ is a $2$-group of a very specific kind. For $n$ a power of $2$, we denote by $\mathcal{R}_n$ the class of groups $G$ such that $G$ has order $n$ and $G$ admits a generating set consisting of $3$ involutions or of $2$ elements, one of which is an involution. The following lemma is the key remark concerning $\mathcal{R}_n$. 

\begin{lemma}\label{3} Let $G\in\mathcal{R}_{2^{i+1}}$ and let $C$ be a central subgroup of order $2$. Then $G/C\in\mathcal{R}_{2^i}$.
\end{lemma}

As we will now explain, Lemma~\ref{3} can be used to construct $\mathcal{R}_{2^i}$ by induction over $i$. We first explain how to construct all possible $G$'s given $K=G/C\in\mathcal{R}_{2^i}$.

Let $K\in\mathcal{R}_{2^i}$ and compute the second cohomology group $H^2(K,\ZZ_2)$ where $\ZZ_2$ is viewed as the trivial $K$-module. For every $2$-cocycle of $H^2(K,\ZZ_2)$, we can use the built-in \texttt{Magma} command to compute the corresponding extension $G$ (with a  power-commutator-presentation) and then check whether $G\in\mathcal{R}_{2^{i+1}}$. By Lemma~\ref{3}, letting $K$ run over the groups in $\mathcal{R}_{2^i}$ and repeating this procedure yields all the groups in $\mathcal{R}_{2^{i+1}}$. We then test for group isomorphism and keep only one representative of each isomorphism class.

Using this procedure inductively, we can construct $\mathcal{R}_{512}$ and $\mathcal{R}_{1024}$. Finally, we simply use the same procedure we used for other orders but we replace $\mathcal{P}_{512}$ and $\mathcal{P}_{1024}$ by $\mathcal{R}_{512}$ and $\mathcal{R}_{1024}$, respectively.

We end this section with a few comments regarding the computational difficulty of our method of computing $\mathcal{R}_{512}$ and $\mathcal{R}_{1024}$. 
In general, given two groups of order $1024$ it is extremely difficult to check whether they are isomorphic (some interesting examples are in~\cite{ELO}). As before, our case is easier because $G$ is $3$-generated. Computing $H^2(H,\ZZ_2)$ for a $2$-group of Frattini class $2$ is again extremely difficult in general: typically  $H^2(H,\ZZ_2)$ is a vector space of very large dimension. In our case, because $H$ is $3$-generated, it follows from~\cite[Corollary~2.2]{Neumann} that this dimension is more modest.

\section{$m=2$ (Tetravalent arc-transitive graphs)}
Our goal in this section is to construct all cubic graphs $\Gamma$ of order at most $1280$ which admit a vertex-transitive group $G$ such that $G_v$ has exactly two orbits on $\Gamma(v)$. In Section~\ref{LocallyC2}, we will construct an auxiliary tetravalent arc-transitive graph and show that $\Gamma$ can be recovered from it. This will reduce our problem to determining all tetravalent arc-transitive graphs of a certain type up to a certain order. We will then show how to accomplish this in Section~\ref{4valent}.

\subsection{Locally-$\ZZ_2^{[3]}$ pairs}\label{LocallyC2}
We first introduce the following definition.

\begin{definition} \label{def:locally}
Let $L$ be a permutation group, let $\Gamma$ be a $G$-vertex-transitive graph and let $v$ be a vertex of $\Gamma$. We denote by $G_v^{\Gamma(v)}$ the permutation group induced by the stabiliser $G_v$ of the vertex $v\in \V\Gamma$ on the neighbourhood $\Gamma(v)$. If $G_v^{\Gamma(v)}$ is permutation isomorphic to $L$, then we say that $(\Gamma,G)$ is \emph{locally-$L$}. 
\end{definition}

Let $\Gamma$ be a cubic $G$-vertex-transitive graph. The case which we called $m=2$ in Section~\ref{sec:Intro} (when $G_v$ has exactly two orbits on $\Gamma(v)$) is equivalent to requiring that the pair $(\Gamma,G)$ be locally-$\ZZ_2^{[3]}$ (where $\ZZ_2^{[3]}$ denotes the permutation group of order $2$ and degree $3$). It thus suffices to determine all cubic graphs of order at most $1280$ which admit a locally-$\ZZ_2^{[3]}$ vertex-transitive group.

We will also need the concept of an arc-transitive cycle decomposition, which was studied in some detail in~\cite{CycleDec}.
\begin{definition}\label{def:cycledec}
A \emph{cycle} in a graph is a connected regular subgraph of valency 2. A \emph{cycle decomposition} $\cC$ of a graph $\Gamma$ is a set of cycles in $\Gamma$ such that each edge of $\Gamma$ belongs to exactly one cycle in $\cC$. If there exists an arc-transitive group $G$ of automorphisms of $\Gamma$ that maps every cycle of $\cC$ to a cycle in $\cC$, then $\cC$ will be called \emph{$G$-arc-transitive}.
\end{definition}

Our main tool in this section is the following construction. The important facts about it will be proved in Theorem~\ref{MergeTheorem}.

\begin{construction}\label{cons:merge}
The input of this construction is a locally-$\ZZ_2^{[3]}$ pair $(\Gamma,G)$, where $\Gamma$ is a cubic graph. The output is a graph $\Merge(\Gamma,G)$ and a partition $\cC(\Gamma,G)$ of the edges of $\Merge(\Gamma,G)$. (It will be shown in Theorem~\ref{MergeTheorem} that, under some mild conditions, $\Merge(\Gamma,G)$ is tetravalent and $\cC(\Gamma,G)$ is a $G$-arc-transitive cycle decomposition of $\Merge(\Gamma,G)$).

Clearly, $G_v$ fixes exactly one neighbour of $v$ and hence each vertex $u\in\V(\Gamma)$ has a unique neighbour (which we will denote $u'$) with the property that $G_u=G_{u'}$. Observe that, for every $g\in G$ and every $v\in\V(\Gamma)$, we have $v''=v$ and $(v')^g=(v^g)'$. It follows that the set $\{\{v,v'\} \colon v\in\V(\Gamma)\}$ (which we will denote $\cT$) is a $G$-edge-orbit forming a perfect matching of $\Gamma$.

Furthermore, since $G$ is vertex-transitive and $G_v$ has two orbits on $\Gamma(v)$ (one of them being $\{v'\}$ and the other one being $\Gamma(v)\setminus\{v'\}$), $G$ has exactly two arc-orbits, and, since $G$ is not edge-transitive, $G$ also has exactly two edge-orbits (one of them being $\cT$). Since $\cT$ forms a perfect matching, the other edge-orbit (which we will call $\cR$) induces a subgraph isomorphic to a disjoint union of cycles, say $C_1,\ldots,C_n$. Let $\cP=\{C_1,\ldots,C_n\}$.

We define a new graph $\Merge(\Gamma,G)$, with vertex-set $\cT$ and two elements $\{u,u'\}$ and $\{v,v'\}$ of $\cT$ adjacent if and only if there is an edge in $\Gamma$ between  $\{u,u'\}$ and $\{v,v'\}$; that is, if and only if there is a member of $\{u,u'\}$ adjacent to a member of $\{v,v'\}$ in $\Gamma$. (We remark that this graph is the quotient graph of $\Gamma$ with respect to the vertex-partition $\cT$.)

Finally, let $\iota$ be the map 
$$
 \iota\colon\{u,v\} \mapsto \{\{u,u'\},\{v,v'\}\}
$$
from $\cR$ to the edge-set of $\Merge(\Gamma,G)$ and let $\cC(\Gamma,G)=\{\iota(C_1),\ldots,\iota(C_n)\}$.
\EOC
\end{construction}

Let $(\Gamma,G)$ be a locally-$\ZZ_2^{[3]}$ pair and assume the terminology of Construction~\ref{cons:merge}. The pair $(\Gamma,G)$ is called \emph{degenerate} if for some elements  $\{u,u'\}$ and $\{v,v'\}$ of $\cT$, there is more than one edge in $\Gamma$ between $\{u,u'\}$ and $\{v,v'\}$. We will show in Lemma~\ref{lemma:degenerate} that, in this case, $\Gamma$ belongs to one of two very specific families of graphs which we now define.

\begin{definition}
A \emph{circular ladder graph} is a graph isomorphic to the Cayley graph $\Cay(\ZZ_n\times\ZZ_2,\{(0,1),(1,0),(-1,0)\})$ for some $n\geq 3$. A \emph{M\"{o}bius ladder graph} is a graph  isomorphic to the Cayley graph $\Cay(\ZZ_{2n},\{1,-1,n\})$ for some $n\geq 2$ (note that we consider the complete graph on $4$ vertices to be a M\"{o}bius ladder).
\end{definition}

\begin{lemma}\label{lemma:degenerate}
If $(\Gamma,G)$ is a degenerate locally-$\ZZ_2^{[3]}$ pair, then $\Gamma$ is either a circular ladder graph or a M\"{o}bius ladder graph.
\end{lemma}
\begin{proof}
Assume the notation of Construction~\ref{cons:merge}. Suppose that there exists an edge $\{u,u'\}$ of the perfect matching $\cT$ contained in a $3$-cycle $(u,u',v)$ of $\Gamma$. Then $\Gamma(v)=\{u,u',v'\}$ and, since $G$ acts transitively on $\cT$, $\{v,v'\}$ is also contained in a $3$-cycle. Since $\Gamma(v) \setminus \{v'\} = \{u,u'\}$, that $3$-cycle contains one of $u$ or $u'$. However, there exists an element of $G_v$ which acts nontrivially on $\Gamma(v)$ and hence fixes $(v,v')$ and swaps $u$ and $u'$. By applying this element, we see that the subgraph induced on $\{v,v',u,u'\}$ is a complete graph, and thus $\Gamma$ is isomorphic to the complete graph on $4$ vertices, a M\"{o}bius ladder graph.

We may thus assume that no edge of $\cT$ is contained in a $3$-cycle. Since $(\Gamma,G)$ is degenerate, there exist two edges $\{u,u'\}$ and $\{v,v'\}$ of $\cT$ such that there are two edges of $\Gamma$ between $\{u,u'\}$ and $\{v,v'\}$. Since no edge of $\cT$ is contained in a $3$-cycle, it follows that the subgraph of $\Gamma$ induced on $\{v,v',u',u\}$ is a $4$-cycle $C$. Call a $4$-cycle of $\Gamma$ \emph{special} if it contains two edges of $\cT$. Then $C$ is special and every special $4$-cycle is chordless. Let $e$ be an edge of $\cR$ contained in $C$. Clearly, $C$ is the only special $4$-cycle containing $e$. Since $\cR$ is a $G$-edge-orbit, it follows that each edge of $\cR$ is contained in exactly one special $4$-cycle and hence each edge of $\cT$ is contained in exactly two special $4$-cycles.

Let $2n$ be the order of $\Gamma$. Since $\Gamma$ is connected, we may label the vertices of $\Gamma$ with the set $\ZZ_n\times\ZZ_2$ in such a way that for every $i\in\ZZ_n$, the set $\{(i,0),(i,1)\}$ forms an edge of $\cT$ and for every $i\in\{1,\ldots,n-1\}$ and every $j\in\ZZ_2$, the set $\{(i,j),(i+1,j)\}$ forms an edge of $\cR$. The vertex labelled $(n,0)$ must be adjacent to either the vertex labelled $(1,0)$ or the one labelled $(1,1)$. In the former case, $\{(n,1),(1,1)\}$ forms an edge of $\Gamma$ and $\Gamma$ is a circular ladder graph, while in the latter case, $\{(n,1),(1,0)\}$ forms an edge of $\Gamma$ and $\Gamma$ is a M\"{o}bius ladder graph.
\end{proof}

We will therefore henceforth assume that $(\Gamma,G)$ is a non-degenerate locally-$\ZZ_2^{[3]}$ pair. Theorem~\ref{MergeTheorem} collects the facts we need about Construction~\ref{cons:merge} in this case.

\begin{theorem}\label{MergeTheorem}
Let $(\Gamma,G)$ be a locally-$\ZZ_2^{[3]}$ pair that is not degenerate and let $\Merge(\Gamma,G)$ and $\cC(\Gamma,G)$ be the output of Construction~\ref{cons:merge} applied to $(\Gamma,G)$. Then $\Merge(\Gamma,G)$ is a connected tetravalent $G$-arc-transitive graph with order half that of $\Gamma$. Moreover, $\cC(\Gamma,G)$ determines a $G$-arc-transitive cycle decomposition of $\Merge(\Gamma,G)$.
\end{theorem}
\begin{proof}
Assume the terminology of Construction~\ref{cons:merge} and let $\Lambda=\Merge(\Gamma,G)$. It is easy to see that, since $\Gamma$ is connected, so is $\Lambda$. Since $\cT$ is a perfect matching of $\Gamma$, its cardinality (and hence the order of $\Lambda$) is half the order of $\Gamma$. Moreover, since $(\Gamma,G)$ is non-degenerate, for each edge $\{\{u,u'\},\{v,v'\}\}$ of $\Lambda$, there is a unique edge of $\Gamma$ (indeed in $\cR$) between $\{u,u'\}$ and $\{v,v'\}$. Conversely, each edge $\{u,v\}$ in $\cR$ induces a unique edge $\{\{u,u'\},\{v,v'\}\}$ of $\Lambda$. In particular, $\iota$ is a bijection. It is trivial to verify that, for every $g\in G$, we have $\iota(\{u,v\}^g) = \iota(\{u,v\})^g$, implying that $\iota$ is an isomorphism between the actions of $G$ on $\cR$ and on the edge-set of $\Lambda$. In particular, since $G$ preserves the edge-set of $\Lambda$, the elements of $G$ can be viewed as automorphisms of $\Lambda$.

The same argument clearly applies to arcs; that is, if we denote by $\vec{\cR}$ the set of underlying arcs of edges contained in $\cR$, then the function
$$
 \vec\iota\colon (u,v) \mapsto (\{u,u'\},\{v,v'\})
$$
is a bijection from $\vec{\cR}$ to the arc-set of $\Lambda$. Again, $\vec\iota$ is an isomorphism between the actions of $G$ on $\vec{\cR}$ and on the arc-set of $\Lambda$. Since $G$ acts transitively on $\vec\cR$, it follows that $\Lambda$ is $G$-arc-transitive.  In particular, $\Lambda$ is regular and a simple counting of the edges yields that it must be tetravalent.

We now show that the action of $G$ on $\cT$ (and thus on the vertices of $\Lambda$) is faithful. Suppose that the kernel of this action contains a non-trivial element $g$. Then $v^g=v'$ for some vertex $v$ of $\Gamma$. Let $e=\{u,v\}$ be an edge of $\cR$ incident with $v$. Then $e^g$ is an edge between $\{u,u'\}$ and $\{v,v'\}$ distinct from $e$. It follows that the pair $(\Gamma,G)$ is degenerate, contrary to our assumption. We conclude that $G$ acts faithfully on $\cT$ and hence $G\leq\Aut(\Lambda)$.

Observe that if two edges $e_1,e_2\in \cR$ are adjacent in $\Gamma$, then $\iota(e_1)$ and $\iota(e_2)$ are adjacent in $\Lambda$. It follows that $\iota(C_1),\ldots,\iota(C_n)$ forms a decomposition of the edge-set of $\Lambda$ into cycles. Note that $G$ preserves the partition $\cP$ of $\cR$ and hence $G$ also preserves the decomposition $\cC(\Gamma,G)=\{\iota(C_1),\ldots,\iota(C_n)\}$. In particular, $\cC(\Gamma,G)$ is a $G$-arc-transitive cycle decomposition of $\Lambda$.
\end{proof}

It will be shown in Theorem~\ref{theo:split} that the following construction is in some sense (which will be made precise) a left-inverse of Construction~\ref{cons:merge}.

\begin{construction}\label{cons:split}
The input of this construction is a pair $(\Lambda,\cC)$, where $\Lambda$ is a tetravalent arc-transitive graph and $\cC$ is an arc-transitive cycle decomposition of $\Lambda$. The output is the graph $\Split(\Lambda,\cC)$, the vertices of which are the pairs $(v,C)$ where $v\in\V(\Lambda)$, $C\in\cC$ and $v$ lies on the cycle $C$, and two vertices $(v_1,C_1)$ and $(v_2,C_2)$ are adjacent if and only if either $C_1\neq C_2$ and $v_1=v_2$, or $C_1=C_2$ and $\{v_1,v_2\}$ is an edge of $C_1$.
\EOC
\end{construction}

\begin{theorem}\label{theo:split}
Let $(\Gamma,G)$ be a locally-$\ZZ_2^{[3]}$ pair that is not degenerate, let $\Merge(\Gamma,G)$ and $\cC(\Gamma,G)$ be the output of Construction~\ref{cons:merge} applied to $(\Gamma,G)$ and let $$\Gamma'=\Split(\Merge(\Gamma,G),\cC(\Gamma,G))$$ be the output of Construction~\ref{cons:split} applied to $(\Merge(\Gamma,G),\cC(\Gamma,G))$. Then $\Gamma'\cong\Gamma$.
\end{theorem}

\begin{proof}
Assume the notation and terminology of Construction~\ref{cons:merge} and Construction~\ref{cons:split}. Vertices of $\Gamma'$ are of the form $(\{v,v'\},C)$ for some $v\in\V(\Gamma)$ and $C\in\cP$ such that one of $v$ or $v'$ lies on $C$. 

We first show that exactly one of $v$ or $v'$ lies on $C$. Assume otherwise and recall that the edges of $C$ are in $\cR$ and hence the edge $\{v,v'\}$ is a chord of $C$. As noted in Theorem~\ref{MergeTheorem}, $G$ acts transitively on $\vec{\cR}$. In particular, there exists $g\in G$ which acts as a one-step rotation of the cycle $C$. It follows that $v$ is adjacent to $v^g$ and $v'$ is adjacent to $(v')^g=(v^g)'$ and hence there are at least two edges between $\{v,v'\}$ and $\{v^g,(v^g)'\}$, contradicting the fact that $(\Gamma,G)$ is not degenerate.

Let $\theta$ be the map from $\V(\Gamma')$ to $\V(\Gamma)$ that maps $(\{v,v'\},C)$ to the one of $v$ or $v'$ that lies on $C$. By the previous paragraph, $\theta$ is well-defined. We show that $\theta$ is a graph isomorphism. Since each vertex of $\Gamma$ lies on exactly one element of $\cP$, $\theta$ is a bijection. It remains to show that $\theta$ is a graph homomorphism from $\Gamma'$ to $\Gamma$. Let $e=\{(\{v_1,v_1'\},C_1),(\{v_2,v_2'\},C_2)$ be an edge of $\Gamma'$. By the definition of adjacency in $\Gamma'$, there are two cases to consider. The first case is when $C_1\neq C_2$ and $\{v_1,v_1'\}=\{v_2,v_2'\}$ . This implies that $\theta(e)=\{v_1,v_1'\}$ and hence $\theta(e)$ is an edge of $\Gamma$. The second case is when $C_1=C_2$ and $\{\{v_1,v_1'\},\{v_2,v_2'\}\}$ is an edge of $C_1$. By definition of adjacency in $\Merge(\Gamma,G)$, this implies that, in $\Gamma$, there is an edge of $C_1$ between the sets $\{v_1,v_1'\}$ and $\{v_2,v_2'\}$ and again in this case it follows that $\theta(e)$ is an edge of $\Gamma$. This concludes the proof that $\theta$ is an isomorphism
\end{proof}

Combining Lemma~\ref{lemma:degenerate}, Theorem~\ref{MergeTheorem} and Theorem~\ref{theo:split} together yields the following result which summarises this section.

\begin{corollary}\label{cor:main}
Let $\Gamma$ be a cubic $G$-vertex-transitive graph of order at most $2n$ such that $G_v$ has exactly two orbits on $\Gamma(v)$. Then either $\Gamma$ is a circular ladder graph or a M\"{o}bius ladder graph, or $\Gamma$ is isomorphic to $\Split(\Lambda,\cC)$, where $\Split(\Lambda,\cC)$ is the output of Construction~\ref{cons:split} applied to some $(\Lambda,\cC)$ where $\Lambda$ is a $G$-arc-transitive tetravalent graph of order at most $n$ and $\cC$ is a $G$-arc-transitive cycle decomposition of $\Lambda$.
\end{corollary}

By Corollary~\ref{cor:main}, to obtain all cubic graphs of order at most $1280$ which admit a locally-$\ZZ_2^{[3]}$ vertex-transitive group of automorphisms, it suffices to determine all the pairs $(\Lambda,\cC)$ where $\Lambda$ is an arc-transitive tetravalent graph of order at most $640$ and $\cC$ is an arc-transitive cycle decomposition of $\Lambda$. Determining the latter is the subject of our next section.

\subsection{Arc-transitive locally imprimitive tetravalent graphs}\label{4valent}

Let $\Lambda$ be an arc-transitive tetravalent graph of order at most 640, let $v$ be a vertex of $\Lambda$ and let $\cC$ be an arc-transitive cycle decomposition of $\Lambda$. By Definition~\ref{def:cycledec}, there exists an arc-transitive group of automorphisms of $\Lambda$ preserving the cycle decomposition $\cC$. Let $G$ be such a group. By~\cite[Theorem 4.2]{CycleDec}, $G_v^{\Lambda(v)}$ is permutation isomorphic to $\ZZ_2\times\ZZ_2$, $\ZZ_4$ or $\D_4$ (the dihedral group of degree $4$).  Observe that this implies that $G_v$ is a $2$-group. 

Suppose first that $G_v^{\Lambda(v)}$ is permutation isomorphic to $\ZZ_2\times\ZZ_2$. Using the connectivity of $\Lambda$, it is easy to see that $G_v$ acts faithfully and regularly on $\Lambda(v)$, and thus $G_v\cong \ZZ_2\times \ZZ_2$. Let $x$ and $y$ be generators of $G_v$. Since $\Lambda$ is $G$-arc-transitive, $G$ is generated by $G_v$ and any element (say $a$) interchanging $v$ and a neighbour of $v$ (say $u$), that is, $G=\langle x,y,a\rangle$. Since $a^2$ fixes the arc $(u,v)$, we obtain $a^2=1$. It follows that $G$ is generated by three involutions, two of which commute. 
In particular, $\Lambda$ is isomorphic to the coset graph $\Cos(G,\langle x,y\rangle, a)$;
that is, the graph with vertex-set being the set of cosets $G/\langle x, y\rangle$ and edges of the form $\{\langle x, y\rangle g, \langle x, y\rangle a g\}$ for $g\in G$. 

Conversely, if a group $G$ is generated by a triple $(x,y,a)$ satisfying $x^2=y^2=a^2=[x,y]=1$ such that the coset graph $\Lambda=\Cos(G,\langle x,y\rangle, a)$ is tetravalent, then $(\Lambda,G)$ is locally-$(\ZZ_2\times\ZZ_2)$. By~\cite[Theorem 4.2(iii)]{CycleDec}, there are exactly three arc-transitive cycle decompositions of $\Lambda$ preserved by $G$. It can be shown that they can be obtained in the following way. Choose $g$ to be one of $ax$, $ay$ or $axy$ (each choice will give rise to one of the three arc-transitive cycle decomposition) and let $C$ be the cycle induced by the $\langle g\rangle$-orbit containing the vertex $\langle x,y\rangle$ of $\Lambda$. Then the $G$-orbit of $C$ is an arc-transitive cycle decomposition preserved by $G$. 

To obtain  all  locally-$(\ZZ_2\times \ZZ_2)$ pairs $(\Lambda,G)$ with $|\V(\Lambda)|\le 640$ and their arc-transitive cycle decompositions, it thus suffices to determine quadruples $(G,x,y,a)$ such that $G$ is a group of order at most $4\cdot 640$ generated by $\{x,y,a\}$ subject to $x^2=y^2=a^2=[x,y]=1$ and such that the coset graph $\Cos(G,\langle x,y\rangle, a)$ is tetravalent. Such a quadruple is also known as a {\em regular map} and they have been determined up to order $4\cdot 640$. (We thank Marston Conder for providing the list~\cite{conderKlein} of all such quadruples $(G,x,y,z)$.)

We may now assume that  $G_v^{\Lambda(v)}$ is permutation isomorphic to $\ZZ_4$ or $\D_4$. By~\cite[Theorem 4.2(ii)]{CycleDec}, $\cC$ is the unique arc-transitive cycle decomposition of $\Lambda$ preserved by $G$. We now describe how to construct it.

\begin{construction}\label{cons:cycledec}
The input of this construction is a locally-$L$ pair $(\Lambda,G)$ 
with $L\cong \ZZ_4$ or $\D_4$.
The output is the unique arc-transitive cycle decomposition of $\Lambda$ preserved by $G$. 
Observe that the permutation groups $\ZZ_4$ and $\D_4$ each have a unique non-trivial system of imprimitivity. In particular, for every vertex $v$ of $\Lambda$, there is a unique $G_v$-invariant partition of the set of four edges incident with $v$ into two blocks of size 2. For every pair of edges $e_1$ and $e_2$ of $\Lambda$, we will write $e_1R e_2$  if they form one such block for some vertex $v$. Since every edge is $R$-related to exactly two edges, it follows that the equivalence classes of the transitive closure of $R$ are cycles and in fact form an arc-transitive cycle decomposition of $\Lambda$ preserved by $G$. \EOC
\end{construction}

In view of Construction~\ref{cons:cycledec}, we are now left with the task of finding all tetravalent arc-transitive locally-$\ZZ_4$ and locally-$\D_4$ pairs $(\Lambda,G)$ with $\Lambda$ of order at most $640$. This can be done as follows.

If $|G_v|\geq 64$, then, since $|\V(\Lambda)| \le 640$, it follows by~\cite[Theorem 2]{PSV4valent} that $\Lambda$ is isomorphic to  one of $\Gamma_5^+$, $\Gamma_5^-$  or $\C(2,r,s)$ for some $r$ and $s$. (The graphs $\Gamma_5^+$ and $\Gamma_5^-$ are defined in~\cite{PSVGamma}, the graphs $\C(2,r,s)$ in~\cite{PXu}.) The automorphism group of the graphs $\Gamma_5^+$, $\Gamma_5^-$ and $\C(2,r,s)$ are known (see~\cite[Theorem 4.2(v)]{PSVGamma} and~\cite[Theorem 2.13]{PXu}). It can be seen that since $|G_v| \ge 64$, the pair $(\Lambda,\Aut(\Lambda))$ is locally-$\D_4$ and, by~\cite[Theorem 4.2(ii)]{CycleDec}, there is a unique arc-transitive cycle decomposition of $\Lambda$ preserved by $\Aut(\Lambda)$, say $\cC'$. Clearly, $\cC'$ is preserved by $G$ and hence $\cC'=\cC$.  The latter can be obtained by applying Construction~\ref{cons:cycledec} to $(\Lambda,\Aut(\Lambda))$.

We may thus assume that $|G_v|\leq 32$. In~\cite{Djokovic}, it is shown that for some triple $(\tilde{G}, \tilde{L}, y)$
in a row of Table~\ref{amalgams}, there exists an epimorphism $\pi\colon\tilde{G}\to G$, such that $G_v=\pi(\tilde{L})$, 
$\tilde{L}\cap \mathrm{Ker}(\pi)=1$ and $\Gamma=\Cos(G,G_v,\pi(y))$. Since $\Gamma$ has order at most $640$, $G$ is a quotient of $\tilde{G}$ by a normal subgroup of index at most $640\cdot|\tilde{L}|$.

\begin{table}[hhh]
\begin{center}
\begin{small}
\begin{tabular}{|c|c|c|}
\hline
\phantom{$\overline{\overline{G_j^G}}$}
$\tilde{G}$ & $\tilde{L}$ &$|\tilde{L}|$ \\
\hline\hline
\begin{tabular}{l}
\phantom{$\overline{\overline{G_j^G}}$}
$\la a,y \mid a^4, y^2 \ra$ \\
\end{tabular} &
$\la a \ra$ &
$4$ \\
\hline
\begin{tabular}{l}
\phantom{$\overline{\overline{G_j^G}}$}
$\la a,b,x,y \mid a^2,b^2,[a,b],x^2,a^xb,y^2,[y,b] \ra$ \\
\end{tabular} &
$\la a,b,x \ra$ &
$8$ \\
\hline
\begin{tabular}{l}
\phantom{$\overline{\overline{G_j^G}}$}
$\la a,b,x,y \mid a^2,b^2,[a,b],x^2,a^xb,y^2b,[y,b] \ra$ \\
\end{tabular} &
$\la a,b,x \ra$ &
$8$ \\
\hline
\begin{tabular}{l}
\phantom{$\overline{\overline{G_j^G}}$}
$\la a,b,c,x,y \mid a^2,b^2,c^2,[a,b],[a,c],[b,c],x^2,a^xc,[x,b], c^xa, y^2, b^yc \ra$ \\
\end{tabular} &
$\la a,b,c,x \ra$ &
$16$ \\
\hline
\begin{tabular}{l}
\phantom{$\overline{\overline{G_j^G}}$}
$\la a,b,c,x,y \mid a^2,b^2,c^2,[a,b],[a,c],[b,c],x^2b,a^xc,[x,b], c^xa, y^2, b^yc \ra$ \\
\end{tabular} &
$\la a,b,c,x \ra$ &
$16$ \\
\hline
\begin{tabular}{l}
\phantom{$\overline{\overline{G_j^G}}$}
$\la a,b,c,x,y \mid a^2,b^2,c^2,[a,b],[a,c]b,[b,c],x^2,a^xc,[x,b], c^xa, y^2, b^yc \ra$ \\
\end{tabular} &
$\la a,b,c,x \ra$ &
$16$ \\
\hline
\begin{tabular}{l}
\phantom{$\overline{\overline{G_j^G}}$}
$\la a,b,c,x,y \mid a^2,b^2,c^2,[a,b],[a,c]b,[b,c],x^2b,a^xc,[x,b], c^xa, y^2, b^yc \ra$ \\
\end{tabular} &
$\la a,b,c,x \ra$ &
$16$ \\
\hline
\begin{tabular}{ll}
\phantom{$\overline{\overline{G_j^G}}$}
$\la a,b,c,d,x,y \mid$ & \hskip-2mm $a^2,b^2,c^2,d^2,[a,b],[a,c],[b,c],[b,d],[c,d], [a,d],$ \\
                                &  \hskip-2mm $x^2,a^xd,b^xc, y^2, b^yd,[c,y],d^yb \ra$ \\
\end{tabular} &
$\la a,b,c,d,x \ra$ &
$32$ \\
\hline
\begin{tabular}{ll}
\phantom{$\overline{\overline{G_j^G}}$}
$\la a,b,c,d,x,y \mid$ & \hskip-2mm $a^2,b^2,c^2,d^2,[a,b],[a,c],[b,c],[b,d],[c,d], [a,d] ,$ \\
                                &  \hskip-2mm $x^2,a^xd,b^xc, y^2c, b^yd,[c,y],d^yb \ra$ \\
\end{tabular} &
$\la a,b,c,d,x \ra$ &
$32$ \\
\hline
\begin{tabular}{ll}
\phantom{$\overline{\overline{G_j^G}}$}
$\la a,b,c,d,x,y \mid$ & \hskip-2mm $a^2,b^2,c^2,d^2,[a,b],[a,c],[b,c],[b,d],[c,d], [a,d]bc,$ \\
                                &  \hskip-2mm $x^2,a^xd,b^xc, y^2, b^yd,[c,y],d^yb \ra$ \\
\end{tabular} &
$\la a,b,c,d,x \ra$ &
$32$ \\
\hline
\begin{tabular}{ll}
\phantom{$\overline{\overline{G_j^G}}$}
$\la a,b,c,d,x,y \mid$ & \hskip-2mm $a^2,b^2,c^2,d^2,[a,b],[a,c],[b,c],[b,d],[c,d], [a,d]bc,$ \\
                                &  \hskip-2mm $x^2,a^xd,b^xc, y^2c, b^yd,[c,y],d^yb \ra$ \\
\end{tabular} &
$\la a,b,c,d,x \ra$ &
$32$ \\
\hline                  
                
\end{tabular}
\end{small}
\caption{Universal groups for locally-$\ZZ_4$ and locally-$\D_4$ graphs}     
\label{amalgams}
\end{center}
\end{table}

To summarise, to find all pairs $(\Gamma,G)$ which are locally-$\ZZ_4$ or locally-$\D_4$ and with $|\V(\Gamma)|\leq 640$ and $|G_v|\leq 32$, it suffices to find all normal subgroups $N$ of the groups $\tilde{G}$ (appearing in Table~\ref{amalgams}) of index at most $640\cdot|\tilde{L}|$.

This task can be achieved by applying the \texttt{LowIndexNormalSubgroups} routine in \texttt{Magma}, which is based on an algorithm of Firth and Holt~\cite{Firth} and which computes all normal subgroups of a finitely presented group up to a given index. However, the current implementation of this algorithm applied to this problem in a straightforward way exceeded the memory capacity of a computer with $24$ GB of RAM. To circumvent the problem, we used the following elementary lemma.

\begin{lemma}\label{MaxOrder}
Let $p$ be a prime and let $G$ be a permutation group of degree $n$ such that each point-stabiliser in $G$ is a $p$-group. Then every element of $G$ has order at most $n$. 
\end{lemma}
\begin{proof}
Let $g$ be an element of $G$ and write $|g|=mp^a$ with $m$ coprime to $p$. Let $s$ be the length of an orbit of $\la g\ra$.
Since $g^s$ fixes at least one point, its order is a power of $p$ and hence $m$ divides $s$. In particular, $s=mp^b$ for some $b\leq a$. Since $|g|$ is the least common multiple of the lengths of all the orbits of $\la g\ra$, we conclude that $|g|$ is equal to the length of the longest orbit of $\la g\ra$ and therefore at most $n$.
\end{proof}

Lemma~\ref{MaxOrder} allowed us to split the computation according to the order of the element $\pi(xy)\in G$. Since $\Lambda$ has order at most $640$ and $G_v$ is a $2$-group, it follows from Lemma~\ref{MaxOrder} that $\pi(xy)$ has order at most $640$ and, in particular, $(\pi(xy))^m=1$ for some $m\leq 640$. It follows that each homomorphic image of $\tilde{G}$ of order at most $640$ is in fact a homomorphic image of the group obtained by adding the relation $(xy)^m$ to $\tilde{G}$ for some $m\leq 640$. It turns out that homomorphic images of these groups can be easily found using \texttt{LowIndexNormalSubgroups}.

This completes the description of the method we used to describe all the pairs $(\Lambda,\cC)$ where $\Lambda$ is an arc-transitive tetravalent graph of order at most $640$ and $\cC$ is an arc-transitive cycle decomposition of $\Lambda$. In particular, we obtain all locally-$L$ pairs $(\Lambda,G)$ where $\Lambda$ is $G$-arc-transitive tetravalent graph of order at most $640$ and $L$ is one of $\ZZ_2\times\ZZ_2$, $\ZZ_4$ or $\D_4$. By combining this with the census of $2$-arc-transitive tetravalent graphs of small order~\cite{Potocnik}, we obtained a list of all tetravalent arc-transitive graphs of order at most $640$ which can be found online~\cite{Census4}.


\end{document}